\documentclass[preprint,12pt]{elsarticle}

\usepackage{amssymb}
\usepackage{amsmath}
\usepackage{amsthm}

\usepackage{amsfonts, amsmath}
\usepackage{mathtools, float}
\usepackage{tikz}
\usepackage{enumitem}
\usepackage{hyperref}
\usepackage[normalem]{ulem}
\usepackage{subcaption} 
\usetikzlibrary{graphs,graphs.standard,quotes,shapes.geometric, backgrounds}
\usepackage{xcolor}

\tikzset{
    tdvert/.style = {circle, draw = black, fill = TDVertColor},
    vert/.style = {circle, draw = black, fill = black}
}
\definecolor{TDVertColor}{RGB}{0, 120, 220}

\newcommand{\tdnum}[1]{\gamma_{t}(#1)}
\newcommand{\tdset}{\gamma_t \text{-set}}
\newcommand{\bondnum}[2]{b_t^{#2}(#1)}
\newcommand{\bond}[1]{b_{t}(#1)}
\newcommand{\set}[1]{\{#1\}}

\usetikzlibrary{graphs,graphs.standard,shapes.geometric}
\tikzset{lvertex/.style={circle,draw=black}}
\tikzset{vertex/.style={circle,draw=black,fill=black}}

\newtheorem{defn}{Definition}
\newtheorem{thm}{Theorem}
\newtheorem{prop}{Proposition}
\newtheorem{lem}{Lemma}
\newtheorem{cor}{Corollary}
\newtheorem*{remark}{Remark}

\newtheoremstyle{case}
  {\topsep} 
  {\topsep} 
  {} 
  {} 
  {\itshape} 
  {:} 
  {.5em} 
  {\thmname{#1} \thmnote{#3}} 
\theoremstyle{case}
\newtheorem{case}{Case}[thm]

\journal{Discrete Mathematics}

\begin{document}

\begin{frontmatter}



\title{The $k$-Total Bondage Number of a Graph}


\author[label1]{Jean-Pierre Appel}
\author[label2]{Gabrielle Fischberg}
\author[label3]{Kyle Kelley}
\author[label1]{Nathan Shank}
\author[label5]{Eliel Sosis} 

\affiliation[label1]{organization={Moravian University, Department of Mathematics and Computer Science},
            addressline={1200 Main St}, 
            city={Bethlehem},
            postcode={18018}, 
            state={PA},
            country={USA}}

\affiliation[label2]{organization={Tufts University, Department of Mathematics},
            addressline={177 College Ave}, 
            city={Medford},
            postcode={02155}, 
            state={MA},
            country={USA}}
            
\affiliation[label3]{organization={Kenyon College, Department of Mathematics},
            addressline={103 College Park Dr}, 
            city={Gambier},
            postcode={43022}, 
            state={OH},
            country={USA}}


\affiliation[label5]{organization={University of Michigan, Department of Mathematics},
            addressline={530 Church St}, 
            city={Ann Arbor},
            postcode={48109}, 
            state={MI},
            country={USA}}

\begin{abstract}
Let $G=(V,E)$ be a connected, finite undirected graph. A set $S \subseteq V$ is said to be a total dominating set of $G$ if every vertex in $V$ is adjacent to some vertex in $S$. The total domination number, $\gamma_{t}(G)$, is the minimum cardinality of a total dominating set in $G$. We define the $k$-total bondage of $G$ to be the minimum number of edges to remove from $G$ so that the resulting graph has a total domination number at least $k$ more than $\gamma_{t}(G)$. We establish general properties of $k$-total bondage and find exact values for certain graph classes including paths, cycles, wheels, complete and complete bipartite graphs. 
\end{abstract}




\end{frontmatter}


\section{Introduction}\label{intro}
The study of domination in graphs began with Berge \cite{berge1958} and Ore \cite{ore1962} and gained momentum through the work of Cockayne and other (\cite{Cockayne1975}, \cite{Cockayne1976}, \cite{Cockayne1977}, and \cite{Cockayne1978}). Domination theory has many applications, including the design and analysis of communication networks (\cite{Du2013}, \cite{haynes1998}, and \cite{Kelleher1985}). For example, the concept of total domination corresponds to placing the smallest number of transmitters in a network so that every network node is adjacent to a transmitter. In addition to coverage it is natural to ask how tolerant such a system is to connection failures, which leads to bondage theory. We can measure fault tolerance in this system by how many failed network connections, or broken edges, the system can have until we need to add additional transmitters. In this paper, we generalize previous concepts in bondage theory and apply them to some common graph classes. 

All graphs that we consider are finite, simple, and undirected. We start with the following definitions on total domination and total bondage in graphs. Total domination was first introduced in \cite{Totaldom1980},  and \cite{TotalDomTextbook} explores in depth many of the results in the field.

\begin{defn}
    A \textbf{total dominating set}, notated TD-set, in a graph $G$ without isolated vertices is a subset $S$ of vertices from $G$ such that every vertex in $V(G)$ is adjacent to at least one vertex in $S$.
    \label{def:tdset}
\end{defn}

\begin{defn}
    A \textbf{minimal TD-set} of $G$ is a TD-set $S$ such that no proper subset of $S$ is a TD-set.
    \label{def:minimaltdset}
\end{defn}

\begin{defn}
    A \textbf{minimum TD-set} of $G$ is a minimal TD-set $S$ such that there exists no TD-set $T$ where $|T|<|S|$.
    \label{def:minimumtdset}
\end{defn}
\begin{defn}
    The \textbf{total domination number}, $\tdnum{G}$, is the cardinality of a minimum total dominating set of $G$.  A TD-set of $G$ of cardinality $\tdnum{G}$ is called a $\tdset$ of $G$.
    
    \label{def:tdnum}
\end{defn}
Note that every graph without isolated vertices has a TD-set, since $S=V(G)$ is such a set, however, we are interested in finding the TD-set of minimum cardinality. See Figure \ref{fig:MinTD-sets} for an example of a minimal TD-set that is not a minimum TD-set. Throughout this paper shaded vertices indicate vertices which form a TD-set and we will only consider graphs with no isolated vertices, even after edges have been removed. 

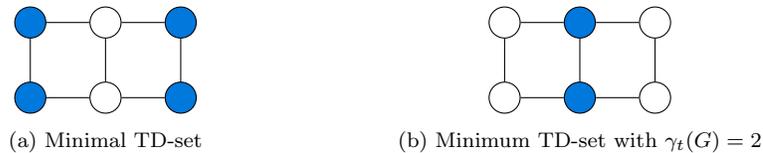
\begin{figure}[h]
    \centering
    \begin{subfigure}{0.45\textwidth}
    \centering
    \begin{tikzpicture}
        \node[tdvert, minimum size = 1em] (a1) at (-4,0) {};
        \node[tdvert, minimum size = 1em] (b1) at (-4,1) {};
        \node[vert, fill=white, minimum size = 1em] (c1) at (-3,0) {};
        \node[vert, fill=white, minimum size = 1em] (d1) at (-3,1) {};
        \node[tdvert, minimum size = 1em] (e1) at (-2,0) {};
        \node[tdvert, minimum size = 1em] (f1) at (-2,1) {};
        \graph[simple] {
            (c1) -- {(a1), (d1), (e1)},
            (d1) -- {(b1), (f1)},
            (a1) -- (b1),
            (e1) -- (f1);          
        };
    \end{tikzpicture}
       \caption{Minimal TD-set}
    \end{subfigure}
    \begin{subfigure}{0.45\textwidth}
    \centering
    \begin{tikzpicture}
        \node[vert, fill=white, minimum size = 1em] (a) at (0,0) {};
        \node[vert, fill=white, minimum size = 1em] (b) at (0,1) {};
        \node[tdvert, minimum size = 1em] (c) at (1,0) {};
        \node[tdvert, minimum size = 1em] (d) at (1,1) {};
        \node[vert, fill = white, minimum size = 1em] (e) at (2,0) {};
        \node[vert, fill=white, minimum size = 1em] (f) at (2,1) {};
        \graph[simple] {
            (c) -- {(a), (d), (e)},
            (d) -- {(b), (f)},
            (a) -- (b),
            (e) -- (f);          
        };
    \end{tikzpicture}
    \caption{Minimum TD-set with $\tdnum{G}=2$}
    \end{subfigure}
    \caption{A minimal TD-set which is not a minimum TD-set.}
    \label{fig:MinTD-sets}
\end{figure}

The minimum number of edge deletions which increases the domination number was first introduced in \cite{bauer1983domination} as \textit{dominating line stability} and was redefined as \textit{bondage number} in \cite{Bondage1990}. A survey of results for bondage number can be found in \cite{bondage_survey}.  Bondage number was recently extended to $k-$synchronous bondage number in \cite{anaya2025}. Similarly we extend the idea of total bondage which was first introduced in \cite{1991Bondage}.
\begin{defn}
    The \textbf{total bondage number}, $\bond{G}$, is the minimum number of edges that must be deleted from $G$ in order to increase the total domination number. 
    \label{def:bondnum}
\end{defn}

\begin{defn}
    The $\mathbf{k}$\textbf{-total bondage number}, $\bondnum{G}{k}$, is the minimum number of edges that must be deleted from $G$ in order to increase the total domination number by at least $k$.
    \label{def:kbondnum}
\end{defn}

\begin{figure}[h]
\captionsetup[subfigure]{margin=0.1\columnwidth}
\begin{center}
    \begin{subfigure}[b]{.3\textwidth}
        \centering
            \begin{tikzpicture}
                \graph[simple, nodes={draw, circle, minimum size =  1em}, empty nodes] 
                {
                    subgraph C_n [name=outer, n=5, clockwise, radius = 0.9cm],
                    outer 1 -- outer 3;
                };
                \node[tdvert, minimum size = 1em] at (outer 1) {};
                \node[tdvert, minimum size = 1em] at (outer 3) {};
            \end{tikzpicture}
        \subcaption{$G$ with $\tdnum{G}=2$\\ \phantom{.}}
        \label{fig:Fig1}
    \end{subfigure}
    \begin{subfigure}[b]{0.3\textwidth}
        \centering
            \begin{tikzpicture}
                \graph[simple, nodes={draw, circle, minimum size =  1em}, empty nodes] 
                {
                    subgraph C_n [name=outer, n=5, clockwise, radius = 0.9cm],
                    outer 5 -- outer 1;
                };
                \node[tdvert, minimum size = 1em] at (outer 1) {};
                \node[tdvert, minimum size = 1em] at (outer 2) {};
                \node[tdvert, minimum size = 1em] at (outer 3) {};
            \end{tikzpicture}
        \subcaption{Subgraph of $G$ with $\bondnum{G}{1} = 1$}   
        \label{fig:Fig2}
    \end{subfigure}
    \begin{subfigure}[b]{0.3\textwidth}
        \centering
            \begin{tikzpicture}
                \graph[simple, nodes={draw, circle, minimum size =  1em}, empty nodes] 
                {
                    subgraph C_n [name=outer, n=5, clockwise, radius = 0.9cm],
                    outer 5 -!- outer 1,
                    outer 1 -- outer 3,
                    outer 3 -!- outer 4,
                };
                \node[tdvert, minimum size = 1em] at (outer 1) {};
                \node[tdvert, minimum size = 1em] at (outer 4) {};
                \node[tdvert, minimum size = 1em] at (outer 3) {};
                \node[tdvert, minimum size = 1em] at (outer 5) {};
            \end{tikzpicture}
        \subcaption{Subgraph of $G$ with $\bondnum{G}{2} = 2$}
        \label{fig:Fig3}
    \end{subfigure}
\caption{Comparison of total bondage numbers, showing deletions need not be successive as $k$ increases.} 
\label{fig:motive_gt(G)}
\end{center}
\end{figure}
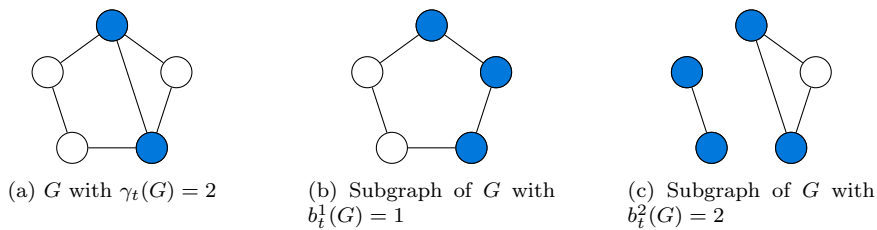

Note that $\bondnum{G}{1} = \bond{G}$. Figure \ref{fig:motive_gt(G)} motivates our definition of $\bondnum{G}{k}$ by showing that $\bondnum{G}{2}$ can involve removing a different set of edges than applying two successive single bondage moves ($\bondnum{G}{1}$). Kulli and Patwari proved results for $\bondnum{G}{1}$ in \cite{1991Bondage} for different types of graph classes. In this paper we extend some of their results to $\bondnum{G}{k}$ for $k>1$.

We first prove some general properties of $\gamma_t(G)$ that help with our investigation of $k$-total bondage. We then consider different graph classes of increasing complexity, and derive results for $\bondnum{G}{k}$ for paths, cycles, wheel graphs, and complete graphs. The main result proves the exact values for $\bondnum{K_n}{k}$ for all possible values of $n$ and $k$. We also provide several bounds for $k$-total bondage for complete bipartite graphs. We finish with some unusual graphs that give general results on $\bondnum{G}{k}$. 

\section{Preliminary Properties}
While in this paper we start with connected graphs, however, graphs often become disconnected when removing edges.
The following obvious lemma relates the total domination number of a graph's components to the entire graph.
\begin{lem}\label{lem:tdnum_components}
    For a disconnected graph $G$ with no isolated vertices and $j$ components, $H_1, H_2, \dots, H_j$, we have 
    \begin{equation*}
        \tdnum{G} = \sum_{i=1}^j \tdnum{H_i}.
    \end{equation*}
\end{lem}

The removal of a single edge can only increase the total domination number by at most 2.  This is shown in the following theorem. Note that we never remove pendant edges because this would create an isolated vertex. 

\begin{thm}\label{thm:max_edge_removal_bondage_diff}
Given a graph $G$ with no isolated vertices and any non-pendant edge $e$, 
\begin{equation*}
\gamma_t(G\setminus \{e\}) \leq \gamma_t(G) + 2.    
\end{equation*}
\end{thm}

\begin{proof}
    Let $G$ be a graph with no isolated vertices and assume there exists a non-pendant edge $ab$.  
    Let $S$ be a $\tdset$ of $G$. Let $c$ be adjacent to $a$ and $d$ be adjacent to $b$, where $c$ and $d$ are not necessarily unique and $c \neq b$ and $d \neq a$. Note $c$ and $d$ will necessarily exist since $ab$ is not a pendant edge. If $a,b \not \in S$, then $S$ is a $\tdset$ of $G \setminus \set{ab}$. If $a \in S$ and $b \not \in S$, then $S \cup \set{b,d}$ is a TD-set of $G \setminus \set{ab}$. If $a,b \in S$, then $S \cup \set{c,d}$ is a TD-set of $G$. \qedhere
\end{proof}
\section{Graph Classes}
Kulli and Patwari (\cite{1991Bondage}) have partial results for the total bondage number for path graphs, cycles, wheel graphs, complete graphs, and complete bipartite graphs. Here we extend these results to $k$-total bondage. 

\subsection{Paths and Cycles}

\indent In \cite{TotalDomTextbook} Henning and Yeo found the total domination number of paths and cycle graphs and in \cite{1991Bondage} Kulli and Patwari found the total bondage number for paths and cycles. Here $n$ denotes the order of the path and cycle. 

\begin{prop}[\cite{TotalDomTextbook}]\label{prop:p01} For any path graph $P_n$ with $n \geq 2$ or cycle graph $C_n$ with $n \geq 3$, 
    \begin{equation*}
        \tdnum{P_n} = \tdnum{C_n} = \begin{cases}
       \frac{n}{2} & \text{ if } n \equiv 0 \pmod{4},\\
        \frac{n+1}{2} & \text{ if } n \equiv 1,3 \pmod{4},\\
        \frac{n}{2} +1 & \text{ if } n \equiv 2 \pmod{4}.\\ 
    \end{cases}
    \end{equation*}
\end{prop}

\begin{thm}[\cite{1991Bondage}] For any path graph $P_n$ with $n \geq 4$, 
    \begin{equation*}
        \bondnum{P_n}{1} = \begin{cases}
        2 & \text{ if } n \equiv 2 \pmod{4},\\
        1 & \text{ otherwise}, \\
        \end{cases} 
    \end{equation*}

and $\bondnum{C_n}{1} = \bondnum{P_n}{1} + 1.$
\end{thm}
Using Lemma \ref{lem:tdnum_components} and Proposition \ref{prop:p01} we can show the following result for the total domination number of the disjoint union of two paths.
\begin{cor}\label{cor:p01}
    Let $P_a$ and $P_b$ be disjoint paths with $a\geq2$ and $b\geq2$. 
    Then, 
    \begin{equation}\nonumber
    \tdnum{P_a + P_b} \leq \tdnum{P_{a+b-2} + P_2}.    
    \end{equation}
\end{cor}
\begin{proof}
    Let $P_a$ and $P_b$ be disjoint paths such that $a\geq2$ and $b\geq2$. We can rewrite the target inequality as $\tdnum{P_a} + \tdnum{P_b} \leq \tdnum{P_{a+b-2}} + 2$ using Lemma \ref{lem:tdnum_components}. We will split this into cases based on $a \pmod{4}$ and $b \pmod{4}$. 
    \begin{case}[1]
        Let $a,b \equiv 0 \pmod{4}$.
                Then $\tdnum{P_a} + \tdnum{P_b} = \frac{a}{2} + \frac{b}{2} = \frac{a+b}{2}$, 
                meanwhile $\tdnum{P_{a+b-2}} + 2 = \frac{a+b-2}{2}+3 = \frac{a+b+4}{2}$. 
                Thus $\tdnum{P_a + P_b} < \tdnum{P_{a+b-2} + P_2}$.
    \end{case}
    \begin{case}[2]
        Let $a, b \equiv 2 \pmod{4}$.
        Then $\tdnum{P_a} + \tdnum{P_b} = \frac{a}{2} + \frac{b}{2} + 2= \frac{a+b+4}{2}$, meanwhile $\tdnum{P_{a+b-2}} +2 = \frac{a+b-2}{2} +3 = \frac{a+b+4}{2}$.
        Thus $\tdnum{P_a + P_b} = \tdnum{P_{a+b-2} + P_2}$. 
    \end{case}
\noindent The remaining cases follow similarly  using Proposition \ref{prop:p01}.
\end{proof}

The following result generalizes Corollary \ref{cor:p01} to the total domination number of any number of paths.
\begin{cor}\label{cor:path_td_sum}
If $\displaystyle n=\sum_{i=1}^{b+1}a_i$ with $a_i\geq 2$, then 
\begin{equation}\nonumber
\sum_{i=1}^{b+1}\gamma_t(P_{a_i})  \leq \gamma_t \left(P_{n-2b} + \sum_{i=1}^{b}P_2\right) = \gamma_t (P_{n-2b}) + 2b. 
\end{equation}
\end{cor}
\begin{proof}
 This follows from Corollary \ref{cor:p01} by induction. \qedhere
\end{proof}

Corollary \ref{cor:path_td_sum} implies that the maximum value of the total domination number for a path of length $n$ with $b$ edges removed can be achieved when edge removals produce a $P_{n-2b}$ along with $b$ disjoint copies of a $P_2$.
We can use this result to find the $k$-total bondage numbers for paths and cycles.

\begin{thm}\label{thm:BtkPaths} For any integer $k \geq 2$ and any $n \geq 2k$, 
\begin{equation*}
\bondnum{P_n}{k} = 
\begin{cases}
    2\left\lfloor \frac{k-1}{2} \right \rfloor +1, & n\equiv 0 \pmod{4},\\
    k, & n \equiv 1,3 \pmod{4},\\
    2\left\lfloor \frac{k-1}{2} \right \rfloor +2, & n\equiv 2 \pmod{4}, 
\end{cases}
\end{equation*}
and $\bondnum{C_n}{k}= \bondnum{P_n}{k}+1.$
\end{thm}
\begin{proof}
    By Corollary \ref{cor:path_td_sum}, to increase $\tdnum{P_n}$ by removing the fewest edges, we will produce disjoint copies of $P_2$.
    Thus for each $n$, we can find the smallest $i$ such that 
\begin{equation}\label{eq:path_tdnum}
\tdnum{P_{n-2i}}+2i - \tdnum{P_n} \geq k.    
\end{equation}
    Applying Proposition \ref{prop:p01} to Equation \ref{eq:path_tdnum} provides the result for $\bondnum{P_n}{k}$. The result for $\bondnum{C_n}{k}$ follows directly from the result for paths since $C_n \setminus \{e\} = P_n$.
\end{proof}

\subsection{Wheel Graphs}
Now we will consider wheel graphs which we define below. 

\begin{defn}
    A wheel graph $W_n$ is a graph on $n+1$ vertices constructed by joining a single vertex and a cycle $C_n$.
\end{defn}

It is easy to see that $\tdnum{W_n} = 2$.  In addition, total bondage number of wheel graphs is given by the following: 
\begin{thm}\label{BondageWheels}(\cite{1991Bondage})
    For any wheel graph $W_n$ with $n \geq 5 $, 
   \begin{equation*}
     \bondnum{W_n}{1} = 2. 
   \end{equation*}
\end{thm}

We now generalize this result by finding the $k$-total bondage for certain values of $k$.  
 
\begin{thm}\label{wheels}
    Consider a wheel $W_n$. For even $n$, the total domination number can increase by at most $n-2$, and 
   \begin{equation*}
     \bondnum{W_n}{n-2} = 2n - \frac{n+4}{2}. 
   \end{equation*}
    For odd $n$, the total domination number can increase by at most $n-1$, and
\begin{equation*}
    \bondnum{W_n}{n-1} = 2n- \frac{n+1}{2}.
\end{equation*}
\end{thm}
\begin{proof}
    Note that $\tdnum{W_n} = 2$ for all $n$. For odd $n$, we can remove edges so that we are just left with $\frac{n+1}{2}$ copies of $P_2$, as shown in Figure \ref{fig:thm3a}. Then all $n+1$ vertices are in the total dominating set. This is the only way all vertices can be in the total dominating set, and $|E(W_n)| = 2n$, so $\bondnum{W_n}{n-1} = 2n - \frac{n+1}{2}$.
    For even $n$, we can remove edges so that we are just left with $\frac{n-2}{2}$ copies of $P_2$ and one $C_3$, as shown in Figure \ref{fig:thm3b}.
    Since $n+1$ is even, not all vertices can be in the total dominating set.
    A $C_3$ allows for the maximum number of edges between the three vertices.
    In this case, we are left with $\frac{n-2}{2}+6 = \frac{n+4}{2}$ edges, so $\bondnum{W_n}{n-2} = 2n - \frac{n+4}{2}$.
\end{proof}

\begin{figure}[h]
    \centering    
    \begin{subfigure}{0.4\textwidth}
    \begin{center}
        \begin{tikzpicture}
            \foreach \i in {1,...,7} {
                \node[tdvert] (\i) at (\i*51.4:2) {};
            }
            \node[tdvert] (0) at (0,0) {};
            \graph {
                (1) -- (2), (3) -- (4), (5) -- (6), (7) -- (0);
            };
        \end{tikzpicture}
        \caption{$\bondnum{W_7}{6}=10$}
        \label{fig:thm3a}
    \end{center}
    \end{subfigure}
    \begin{subfigure}{0.4\textwidth}
    \begin{center}
        \begin{tikzpicture}
            \foreach \i in {1,...,8} {
                \node[tdvert] (\i) at (\i*45:2) {};
            }
            \node[draw, circle] (0) at (0,0) {};
            \graph[nodes={vertex}, empty nodes] {
                (1) -- (2), (3) -- (4), (5) -- (6), (7) --(8),
                (0) -- {(7),(8)};
            };
        \end{tikzpicture}
    \end{center}
        \caption{$\bondnum{W_8}{6}=10$}
        \label{fig:thm3b}
    \end{subfigure}
    
    \caption{Wheels with the maximum increase to their total domination number}
\end{figure}
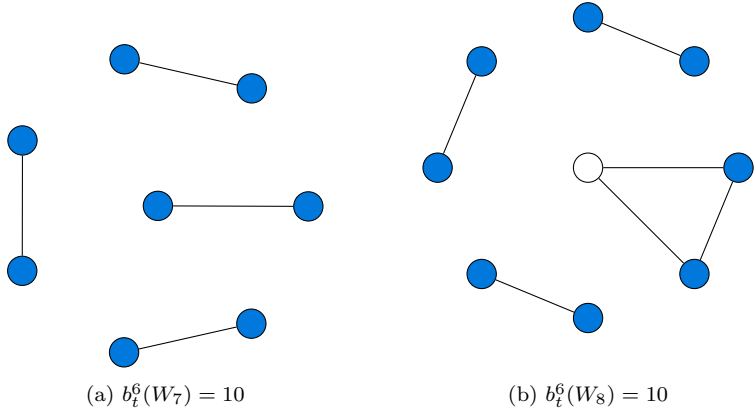
\begin{thm}
    For a wheel $W_n$ where $n \geq 3k$,
    \begin{equation*}
        \bondnum{W_n}{k-1} = k.
    \end{equation*}
\end{thm}   

\begin{proof}
    First we show by construction that $\bondnum{W_n}{k-1} \leq k$. Label the vertices of $W_n$ from $\{v_0, v_1, \ldots ,v_n\}$ such that $\deg(v_0) = n$. Remove the set of $k$ edges $\{v_0v_i \mid i\equiv 0 \pmod{3}, 1 \leq i \leq 3k\}$. The total domination number would then be $k+1$, hence an increase of $k-1$.

    Next we show that when we remove less than $k$ edges, we can always find a TD-set of cardinality at most $k$.
First let $v_0$ be in the TD-set.
For each vertex $v_i$ where we have removed $v_0v_i$, we can choose a neighbor to add to the TD-set since $v_i$ can not be isolated.
When possible, choose a neighbor of $v_i$ that is adjacent to $v_0$ so that $v_0$ is adjacent to another vertex in the TD-set.
If along a sequence of vertices it is never possible to choose a sequence of neighbors to find one adjacent to $v_0$, that means we must have disconnected a component of the graph. Disconnecting this component would require removing two edges from $W_n \setminus \{v_0\}$, so we can choose an arbitrary vertex that is still connected to $v_0$ to add to our TD-set. Using this strategy we have created a TD-set of cardinality at most $k$ after removing less than $k$ edges. 

\end{proof}
\subsection{Complete Graphs}
It is clear that $\gamma_t(K_n) = 2$.  In \cite{1991Bondage} Kulli and Patwari state the following for complete graphs. 

\begin{thm}[\cite{1991Bondage}]\label{thm:Kulli}
    Let $K_n$ be a complete graph with $n\geq 5$ vertices. Then, 
    \begin{equation*}
        \bondnum{K_n}{1} = 2n-5.
    \end{equation*}
\end{thm}
\begin{remark}
    Kulli and Patwari correctly show by construction that $\bondnum{K_n}{1} \leq 2n-5$. They do this by considering a complete graph where $t$, $u$, $v$, and $w$ are four consecutive vertices, and removing edges as shown in Figure \ref{fig:KulliProof}. Suppose that $\bondnum{K_n}{1} < 2n-5$. They claim that then either $\deg(u) \neq 2$, $\deg(v) \neq 2$, or $t$ and $w$ are adjacent, which leads to a contradiction. However, the proof supposes that the method used in Figure \ref{fig:KulliProof} is the only way to increase $\gamma_t$ by removing $2n-5$ edges. However, Figure \ref{fig:KulliContradiction} demonstrates that this is not the case.
\end{remark}

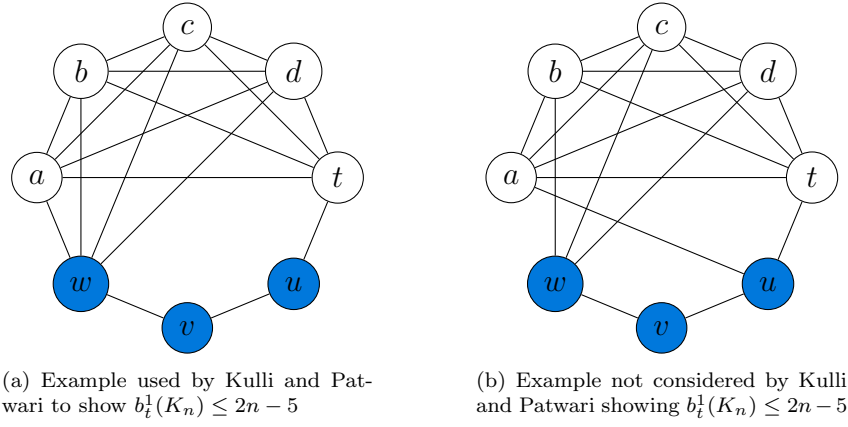
\begin{figure}[h]
\captionsetup[subfigure]{margin=0.1\columnwidth}
    \begin{center}
        \begin{subfigure}[b]{0.45\textwidth}
            \centering
                \begin{tikzpicture}
                    \graph[math nodes, simple, nodes={draw, circle}, clique, n=8, clockwise, radius=2cm]
                        {
                            1/"c", 2/"d", 3/"t", 4/"u"[fill=TDVertColor], 5/"v"[fill=TDVertColor], 6/"w"[fill=TDVertColor], 7/"a", 8/"b";
                            1 -- {2,3,6,7,8};
                            2 -- {1,3,6,7,8};
                            3 -- {1,2,4,7,8};
                            4 -- {3,5};
                            5 -- {4,6};
                            6 -- {8};
                            7 -- {6,8};
                        };
                \end{tikzpicture}
            \subcaption{Example used by Kulli and Patwari to show $\bondnum{K_n}{1}\leq 2n-5$}
            \label{fig:KulliProof}
        \end{subfigure}
        \begin{subfigure}[b]{0.45\textwidth}
            \centering
                \begin{tikzpicture}
                    \graph[math nodes, simple, nodes={draw, circle}, clique, n=8, clockwise, radius=2cm]
                        {
                            1/"c", 2/"d", 3/"t", 4/"u"[fill=TDVertColor], 5/"v"[fill=TDVertColor], 6/"w"[fill=TDVertColor], 7/"a", 8/"b";
                            1 -!- {4,5};
                            2 -!- {4,5};
                            3 -!- {5,6};
                            4 -!- {1,2,6,8};
                            5 -!- {1,2,3,7,8};
                            6 -!- {3,4,7};
                        };
                \end{tikzpicture}
            \subcaption{Example not considered by Kulli and Patwari showing $\bondnum{K_n}{1} \leq 2n-5$ }
            \label{fig:KulliContradiction}
        \end{subfigure}
\caption{Examples with $ n=8$ showing that the removal of a minimum TD-set does not imply existence of two vertices of degree 2 } 
\label{fig:gt(G)}
\end{center}
\end{figure}

We now give a complete proof of Theorem \ref{thm:Kulli}.

\begin{proof}
    By way of contradiction, suppose $\bondnum{K_n}{1} \leq 2n-6$. 
    For the base case, consider $n=5$. Since $\bondnum{K_5}{1}\leq 4$, there exists $S\subset E(K_5)$ such that $|S|=4$ and $\gamma_t(K_5 \setminus S) \geq 3$.
    Then $D(K_5\setminus S)=12$, where $D(G)$ represents the total degree of a graph $G$.
    Note that $K_5$ has $5$ vertices, so there exists $x\in V(K_5\setminus S)$ such that $\deg(x)\geq 3$. 
    If $\deg(x)=4$, then $\{x, a\}$  is a TD-set for any vertex $a$.
    Consider if $\deg(x)=3$, then there exists only one $y \in V(K_5 \setminus S)$ such that $x$ is not adjacent to $y$.
    However, we cannot have any isolated vertices, so there exists $z \in V(K_5 \setminus S)$ such that $y$ is adjacent to $z$.
    Note $x$ must also be adjacent to $z$, so $\{x,z\}$ forms a TD-set. In both cases $\gamma_t(G) = 2$, thus a contradiction when $n=5$.
    Next, let $n$ be the least possible integer such that $\bondnum{K_n}{1} \leq 2n-6$. Then there exists $S\subset E(K_n)$ such that $|S|=2n-6$ and $\gamma_t(K_n \setminus S) \geq 3$.
    Let $G=K_n \setminus S$, so $D(G) = n^2 - 5n +12$. Let $x \in V$ be the vertex of maximum degree in $G$. Therefore, $\deg(x)>n-5$.  We proceed based on the degree of $x$. 
\begin{case}[1]
    If $\deg(x)=n-1$ or $n-2$ then the result follows from a similar argument as the base case.
\end{case}
\begin{case}[2]
    Suppose that $\deg(x)=n-3$. Let $G'=G\setminus \{x\}$, so $D(G')=n^2-7n+18$.
    We know $\bondnum{K_{n-1}}{1} \leq 2n-7$ since $n$ is the least possible integer such that $\bondnum{K_n}{1} \leq 2n-6$.
    So there exists $T \subset E(K_{n-1})$ such that $|T|=2n-7$ and $\gamma_t(K_{n-1}\setminus T) \geq 3$.
    Then $D(K_{n-1}\setminus T) = n^2 - 7n +16.$ Since $D(G') > D(K_{n-1}\setminus T)$, we must have $\gamma_t(G')=2$. Let $\{a,b\}\subset V(G')$ be this TD-set, which implies $a$ is adjacent to $b$.

    If $x$ is adjacent to $a$ or $b$ in $G$, then $\{a,b\}$ is still a TD-set in $G$.

    Suppose $x$ is not adjacent to $a$ or $b$.
    Suppose all $n-3$ other vertices connect to, without loss of generality, vertex $a$. So there exists $c\in V(G)$ such that $c$ is adjacent to both $x$ and $a$. Then $\{a,c\}$ is a TD-set.

    Next suppose $a$ and $b$ are each adjacent to at least one of the remaining $n-3$ vertices, which we will call $a'$ and $b'$ respectively.
    Let $G''=G \setminus \{x,a,b\}$, so $D(G'') \leq n^2 -9n +22$.
    Note that $D(K_{n-3})=n^2-7n+12$, so $G''$ has $n-5$ fewer edges than $K_{n-3}$.
    Let $k$ be the number of vertices in $G''$ that $a$ and $b$ are both adjacent to, then $D(G'') = n^2 -9n +22 - 2k$.
    We can always choose $a'$ and $b'$ to be adjacent, because $D(G'') > D(K_i)+D(K_j)$, for any integers such that $i+j+k=n-3$, as shown below:
       \begin{align*}\label{ijk}
       n^2 -9n +22 - 2k - &D(K_i)-D(K_j) \\& =  n^2 -9n +22 - 2k - i^2 + i - j^2 - j \\
        & = -3n +13 - 2k + i + j + 2ij +2jk +2ik +k^2 \\
        & = k^2 - 5k + 4 -2i - 2j + 2ij +2jk +2ik  \\
        & > 0.
       \end{align*}
    We show one possible case for $G$ and $G''$ in Figure \ref{fig:complete_maxdeg_n-3_3}.
        
        \begin{figure}[h]
            \centering
            \begin{tikzpicture}
                \node[draw, circle] (x) at (0,0) {$x$};
                \node[draw, circle] (a) at (6,2) {$a$};
                \node[draw, circle] (b) at (6,-2) {$b$};
                \draw (1.5,-3) rectangle (4.5,3);
                \node at (3, 3.5) {$G''$};
                \draw (x) -- (1.5,3);
                \draw (x) -- (1.5,-3);
                \node[draw, circle] (a') at (3, 1.5) {$a'$};
                \node[draw, circle] (b') at (3, -1.5) {$b'$};
                \draw (a) -- (b);
                \draw (1.5,0) -- (4.5,0);
                \draw (a') -- (b');
                \draw (a) -- (a');
                \draw (b) -- (b');
            \end{tikzpicture}
            \caption{$G = K_n \setminus S$ with subgraph $G''$}
            \label{fig:complete_maxdeg_n-3_3}
        \end{figure}
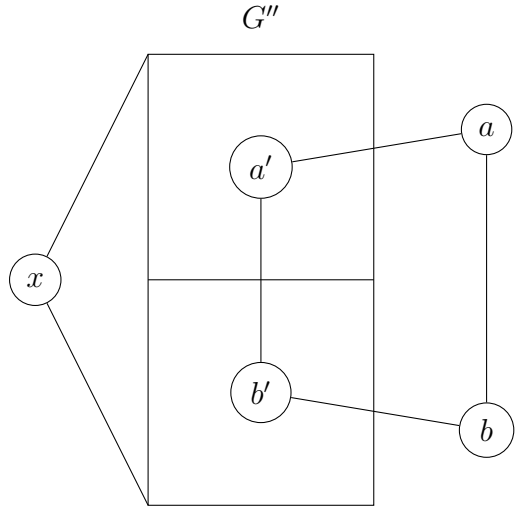

        \begin{figure}[h]
            \centering
            \begin{tikzpicture}[scale=1.25]
                \node[draw, circle] (a') at (6.5,4) {$a'$};
                \node[draw, circle] (b') at (6.5,1) {$b'$};
                \node[draw, circle] (y) at (3.5,1.5) {$y$};
                \draw (0,1.5) rectangle node{$X_0$} (2,3.5);
                \draw (3,0) rectangle node{$X_2$} (5,2);
                \draw (3,3) rectangle node{$X_1$} (5,5);
                \draw[dashed] (a') -- (5,4);
                \draw[dashed] (b') -- (5,4);
                \draw (a') -- (5,1);
                \draw (b') -- (5,1);
                \draw (y) -- (2,2.5);
                \draw (y) -- (4,3);
            \end{tikzpicture}
            \caption{$G''$: For each vertex in $X_1$, only one of the dotted lines is an edge}
            \label{fig:thrm4.2}
        \end{figure}
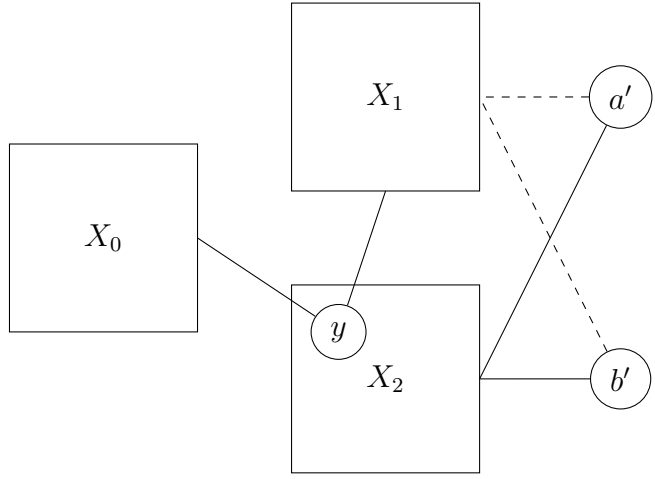
We categorize $V(G'')$ as follows: Let $X_2 \subseteq V(G'')$ be the set of vertices in $G''$ which are adjacent to $a'$ and $b'$, i.e. $X_2 = N(a')\cap N(b')\cap V(G'')$.
Let $X_1 \subseteq V(G'')$ be the set of vertices in $G''$ which are adjacent to exactly one of $a'$ or $b'$.
Let $X_0 \subseteq V(G'')$ be the set of vertices in $G''$ which are not adjacent to either $a'$ or $b'$.
Suppose there exists a $y \in X_2$ which is adjacent to every vertex in $X_0 \cup X_1$ as in Figure \ref{fig:thrm4.2}.
Note that every vertex in $G''$ is adjacent to $a$ or $b$ in $G$.
Assume without loss of generality that $y$ is adjacent to $a$.
Then we can see that $\{y, b'\}$ is a TD-set of $G$.
Therefore every vertex $y \in X_2$ must be missing an edge to a vertex in $X_1\cup X_0$.
Also notice every vertex in $X_1$ is missing an edge to either $a'$ or $b'$ and every vertex in $X_0$ is missing at least the two edges; one to $a'$ and one to $b'$.
Therefore the total number of edges missing in $G''$ is at least $|X_2|+|X_1|+2|X_0|$.
However, $|X_2|+|X_1|+|X_0| = n-5$ and $G''$ has $n-5$ fewer edges than $K_{n-3}$, so $|X_0|=0$.
Then $\{a',b'\}$ is a TD-set of $G$.
\end{case}

\begin{case}[3]
    Suppose there exists no $x\in V(G)$ such that $\deg(x)\geq n-3$.
    Recall that $D(G) = n^2-5n+12=n(n-5)+12$, so there exist at least 12 vertices of degree $n-4$.  Let $x \in V(G)$ such that $deg(x)=n-4$. Let $x$ not be adjacent to $a$, $b$, and $c$, and $\deg(a)=n-4-y_1$, $\deg(b)=n-4-y_2$, and $\deg(c)=n-4-y_3$ for $y_1, y_2,y_3 \geq 0$. We will show that there must exist another vertex adjacent to $a$, $b$, and $c$ that makes a TD-set with $x$. In total, there are $9 + y_1 +y_2 +y_3$ edges removed from $a$, $b$, and $c$. Three of those edges would connect to $x$, so there are at most $6 + y_1 +y_2 +y_3$ other vertices that are not adjacent to at least one of $a$, $b$, or $c$. 
    From $D(G)$ we have that $n \geq 12+ y_1 + y_2 + y_3$, so $\deg(x) = n-4>6 + y_1 +y_2 +y_3 $. Then there must exist a vertex $z$ adjacent to $x$, $a$, $b$, and $c$, so $\{x,z\}$ is a TD-set of $G$. \qedhere
\end{case}
\end{proof}

Next, we show that increasing the total domination number of a complete graph by two only requires removing one additional edge. The set of edges removed, however, is different.
\begin{thm}\label{thm:bt2complete}
    For a complete graph $K_n$ with $n \geq 4$,
\begin{equation}\nonumber
\bondnum{K_n}{2} = 2n-4.
\end{equation}
\end{thm}

\begin{proof}
    We can disconnect a $P_2$ from $K_n$ by removing $2n-4$ edges, so $\bondnum{K_n}{2} \leq 2n-4$.
    Suppose for sake of contradiction that there exists $n$ such that $\bondnum{K_n}{2} \leq 2n -5$, and we will show that after removing $2n-5$ edges we can always find a TD-set of cardinality at most three.
    There must exist a set $S$ such that $\gamma_t(K_n\setminus S) \geq 4$, and $|S|=2n-5$.
    Then $D(K_n\setminus S) = n^2 - 5n +10$, so by letting $x$ be the vertex of maximal degree, we have $\deg(x) \geq n-4$.
    If $\deg(x)=n-1$ or $n-2$, we can find a TD-set using the methods from Theorem \ref{thm:Kulli}. 

    If $\deg(x)=n-3$, then $x$ is not adjacent to two vertices $a$ and $b$.
    If $a$ and $b$ are adjacent, then either $a$ or $b$ must be adjacent to another vertex $y$ since a complete graph will have no disconnected components after removing $2n-5$ edges without isolating a vertex.
    Without loss of generality, suppose $y$ and $a$ are adjacent, then $\{x,y,a\}$ is a TD-set of $K_n\setminus S$.
    Similarly, if $a$ and $b$ are not adjacent, they must each be adjacent to vertices $y$ and $z$, and then $\{x,y,z\}$ is a TD-set of $K_n\setminus S$.  

    If $\deg(x)=n-4$, then $x$ is not adjacent to three vertices $a$, $b$, and $c$. Let $V' = V \setminus \{a, b, c, x\}$. 
    If there is a vertex $z \in V'$ which is adjacent to $a, b$, and $c$ then $\{x, z\}$ is a TD-set. If there is a vertex $z \in V'$ which is adjacent to two vertices in $\{a, b, c\}$, say $a$ and $b$, then $\{x, z, w\}$ is a TD-set for any vertex $w$ which is adjacent to $c$.  
    If every vertex in $V'$ is adjacent to at most one vertex in $\{a, b, c\}$ then 
    \begin{align*}
        D(K_n\setminus S) &= \deg(x) + \deg(a, b, c) + \deg(V')\\
        & \leq n-4 + 6 + n-4 + (n-4)(n-4)\\
        & = n^2 - 6n + 14.        
    \end{align*}
    which is a contradiction because $D(K_n \setminus S) = n^2-5n+10$.
\end{proof}

Using $\bondnum{K_n}{1}$ and $\bondnum{K_n}{2}$ as base cases for induction, we find a general result for $\bondnum{K_n}{k}$.
In our proof we apply the following helpful result from \cite{sanchis2004relating} for when the graph stays connected.

\begin{thm}[\cite{sanchis2004relating}]\label{edgesbound}
    Let $G$ be a connected graph with $n$ vertices and $l$ edges, and $\gamma_t(G) \geq 5$. Then $l \leq \binom{n- \gamma_t(G) +1}{2} + \left\lfloor \frac{\gamma_t(G)}{2} \right\rfloor$. 
\end{thm}
We first prove a lemma that will allow us to use Theorem \ref{edgesbound} in the proof for $\bondnum{K_n}{k}$.
\begin{lem}\label{nonincreasing}
    Let $G$ be a graph with $n$ vertices. Then $f(\gamma_t(G))=\binom{n- \gamma_t(G) +1}{2} + \left\lfloor \frac{\gamma_t(G)}{2} \right\rfloor$ is a nonincreasing function for $2 \leq \gamma_t(G) \leq n$. 
\end{lem}
\begin{proof}
    We can compare the values of the function for $\gamma_t(G)=x$ and $\gamma_t(G)=x+1$ for $2 \leq x \leq n-1$. We have 
    \begin{eqnarray*}
    f(x) & = & \binom{n-x+1}{2}+\left\lfloor\frac{x}{2}\right\rfloor \\
    & = & \frac{(n-x)(n-x-1)}{2} + \frac{2(n-x)}{2} + \left\lfloor\frac{x}{2}\right\rfloor \\
    & \geq & \frac{(n-x)(n-x-1)}{2} + 1 + \left\lfloor\frac{x}{2}\right\rfloor \\
    & \geq & \binom{n-x}{2}+\left\lfloor\frac{x+1}{2}\right\rfloor \\
    & = & f(x+1). 
    \end{eqnarray*} 
\end{proof}

\begin{thm}\label{thm:complete_kbond}
    Let $K_n$ be a complete graph with $n \geq 5$, and $k \leq n-2$ for even $k$ and $k \leq n-3$ for odd $k$. Then,
    \begin{equation*}
        \bondnum{K_n}{k} =
        \begin{cases}
        nk - \frac{k^2+2k}{2} & k \equiv 0 \pmod{2}\\
        nk + n - \frac{k^2 +4k+5}{2}  & k \equiv 1 \pmod{2}.
        \end{cases}
    \end{equation*}
\end{thm}

\begin{proof}
    We can apply the methods described in Theorems \ref{thm:Kulli} and \ref{thm:bt2complete} iteratively to find a bound for $\bondnum{K_n}{k}$.
    The method from Theorem \ref{thm:bt2complete} results in a $K_{n-2}$ and a $P_2$, and can be repeated $i$ times to obtain a $K_{n-2i}$ and $i$ copies of $P_2$.
    This gives a bound for even $k$. For odd $k$ we can repeat the same process $\frac{k-1}{2}$ times, and then use the method from Theorem \ref{thm:Kulli} to increase $\gamma_t$ by one.
    We now show that these bounds achieve equality, meaning it is impossible to increase the total domination number more efficiently than by using these methods.

First, we consider when the graph stays connected after removing edges.
    We only need to consider when $\gamma_t(G) \geq 5$ because we proved the cases $\bondnum{K_n}{1}$ and $\bondnum{K_n}{2}$ separately, which lets us apply Theorem \ref{edgesbound}.
    For even $k$, suppose $\bondnum{K_n}{k}<nk-\frac{k^2+2k}{2}$.
    Then the number of edges remaining after removing $nk-\frac{k^2+2k}{2}-1$ is $l=\frac{n^2-n}{2}-nk+\frac{k^2+2k}{2}+1=\frac{n^2-2nk-n+k^2+2k+2}{2}$. Note $\bondnum{K_n}{k}$ gives us $\gamma_t \geq k+2$ for the resulting graph, and from Lemma \ref{nonincreasing} the number of edges will be maximized for $\gamma_t = k+2$. Then the bound from Theorem \ref{edgesbound} gives us
\begin{eqnarray*}
l & \leq & \binom{n- k-1}{2} + \left\lfloor \frac{k+2}{2} \right\rfloor \\
& = & \frac{(n-k-1)(n-k-2)}{2} + \frac{k+2}{2} \\
& = & \frac{n^2-2nk-3n+k^2+4k+4}{2}.
\end{eqnarray*}
Taking the difference between the number of edges and the bound gives  
\begin{align*}
\frac{n^2-2nk-n+k^2+2k+2}{2} - &\frac{n^2-2nk-3n+k^2+4k+4}{2}\\
& =  \frac{2n-2k-2}{2} \\
& =  n-k-1.
\end{align*}
Since this is always greater than zero, the number of edges exceeds the bound, so $\gamma_t(G) < k+2$.
    We can take the same approach for odd $k$.
    The number of edges remaining after removing $nk+n-\frac{k^2+4k+5}{2}-1$ is $l=\frac{n^2-n}{2}-nk-n+\frac{k^2+4k+5}{2}+1=\frac{n^2-2nk-3n+k^2+4k+7}{2}$.
    The bound from Theorem \ref{edgesbound} gives us
\begin{eqnarray*}
l & \leq & \frac{n^2-2nk-3n+k^2+4k+3}{2}.
\end{eqnarray*}
Taking the difference between the number of edges and the bound gives  
\begin{eqnarray*}
\frac{n^2-2nk-3n+k^2+4k+7}{2} - \frac{n^2-2nk-3n+k^2+4k+3}{2} =  2.
\end{eqnarray*}
    Since this is greater than zero, the number of edges exceeds the bound, so $\gamma_t(G) < k+2$.
    This proves the result in the case when the graph stays connected.

    Next, we look at when the graph is split into multiple components by removing edges.
    Suppose for induction on $k$ that for all $n$, and even $k \leq n-2$, we have $\bondnum{K_n}{k} = 2 \sum_{i=1}^\frac{k}{2} (n-2i)$.
    Now consider $\bondnum{K_n}{k+2}$.
    If $k+2 \leq n -2$, and the process for increasing $\gamma_t$ involves removing a $P_2$, then the result follows by induction using Theorems \ref{thm:Kulli} and \ref{thm:bt2complete}.
    Note that $\bondnum{K_n}{k+2}$ is not well defined for $k+2 > n -2$.
    We can use the same approach using induction for odd $k$ in the case of removing a $P_2$.

    Finally, consider removing a component of the graph of order $m>2$.
    We only need to consider removing a single component, since this would imply that then any number of components follows by induction.
    Therefore, we start by disconnecting a $K_m$ component.
    Fix $k$, and suppose the theorem holds for all $x \leq k$ and $n \geq x+2$.
    We want to show that the theorem holds for $k+2$ if $k+2 \leq n-2$. 
    Since we are using strong induction, we can apply the theorem to the $K_m$ and $K_{n-m}$ components of the graph.
    We want to show that
\begin{equation}\label{eq:components}
    m(n-m) + \bondnum{K_m}{i} + \bondnum{K_{n-m}}{k-i-2} \geq \bondnum{K_n}{k},
\end{equation}
    where $0 \leq i \leq m-2$. To prove this, we consider several cases depending on the parity of $k$ and $i$. 

    First, let $k$ and $i$ both be even.
    Note that 
\begin{align*}
&m(n-m) + \bondnum{K_m}{i} + \bondnum{K_{n-m}}{k-i-2}-\bondnum{K_n}{k}\\
&= mn-m^2 + mi - \frac{i^2}{2}-i + (n-m)(k-i-2) \\
& \qquad - \frac{(k-i-2)^2}{2} - k + i +2 - nk - \frac{k^2}{2}-k \\
&= mn-m^2 + mi - \frac{i^2}{2}-i + (n-m)(k-i-2) \\
& \qquad - \frac{(k-i-2)^2}{2} - k + i +2 -  nk - \frac{k^2}{2}-k \\
&=mn-m^2 + mi - i^2 -2n -ni - mk + 2m +mi -2+ki+2k-2i+2 \\
&=(m-i-2)(n-m-k+i).
\end{align*}
    Note that $m-i-2 \geq 0$ and $n - m -k+i \geq 0$ from the bounds for increasing the total domination number for a complete graph. Therefore Equation \ref{eq:components} holds for even $k$ and $i$.

    We can follow the same line of reasoning to show that Equation \ref{eq:components} holds when $k$ is even and $i$ is odd, and when $k$ and $i$ are odd.
    When $k$ is even and $i$ is odd, we need to show 
\begin{equation*}
(m-i-3)(n-m-k+i) +m-i+2 \geq 0,     
\end{equation*}
which is true since when $i$ is odd we have $m-i-3 \geq 0$. When $k$ and $i$ are odd we need to show 
\begin{equation*}
(m-i-3)(n-m-k+i) +k \geq  0,  
\end{equation*}
    which holds for the same reasons as above. This proves the result for when the graph becomes disconnected, which completes the proof. \qedhere
\end{proof} 
Complete graphs require removing more edges than any other graph to increase the total domination number, which leads to the following corollary. 

\begin{cor}
    Let $G$ be a graph on $n \geq 5$ vertices.
    For any value $k \leq n-2 + \frac{(-1)^n-1}{2}$, if $\bondnum{G}{k}$ exists then
    \begin{equation*}
        \bondnum{G}{k} \leq \bondnum{K_n}{k}.
    \end{equation*}
\end{cor}
\subsection{Complete Bipartite Graphs}
In this section we consider complete bipartite graphs $K_{a,b}$ with $2 \leq a \leq b$, and we label the vertices $a_i$ and $b_j$ for $i \in \{1,2, \ldots, a\}$ and $j \in \{1,2, \ldots, b\}$. We can not have $a=1$ because the removal of any edge in a star creates an isolated vertex. It is easy to see that $\gamma_t(K_{a,b}) = 2.$
Kulli and Patwari proved the following theorem by removing the set of edges $\{a_ib_i | 1 \leq i \leq a\}$, as shown in Figure \ref{fig:Kab1}.
We provide an alternate method that generalizes more efficiently when considering $\bondnum{K_{a,b}}{k}$. 
\begin{thm}[\cite{1991Bondage}]\label{KabThm0}
     For a complete bipartite graph with $2 \leq a \leq b$, 
\begin{equation}\nonumber
    \bondnum{K_{a,b}}{1} = a.
\end{equation}
\end{thm}
\begin{proof}
    One could remove the set of edges $\{a_ib_j | 1 \leq i \leq a-1\}$ with fixed $j$, along with $\{a_a b_{\ell}\}$ for fixed $\ell \neq j$.
    This method requires removing $a$ edges to increase the domination number by one. 
\end{proof}

\begin{figure}[h]
    \centering
   \begin{tikzpicture}
            \graph[nodes={draw, circle,minimum size = 1em},
           empty nodes,radius = 1.1cm, branch down=.8 cm,
           grow right sep=4cm] {subgraph I_nm [V={a, b, c}, W={1,...,6}];
  a -- { 2,3,4,5,6};
  b -- { 1,3,4,5,6 };
  c -- { 1,2,4,5,6 };
  
};
\node[tdvert, minimum size = 1em] at (a) {};
\node[tdvert, minimum size = 1em] at (b) {};
\node[tdvert, minimum size = 1em] at (4) {};
        \end{tikzpicture}
    \caption{Kulli and Patwari's method: $\bondnum{K_{3,6}}{1}$}
    \label{fig:Kab1}
\end{figure}
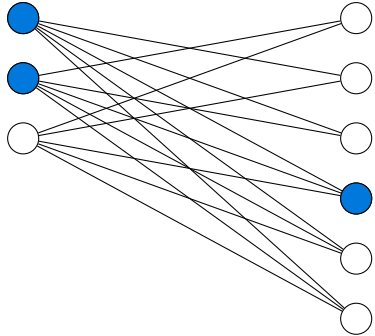

\begin{thm}\label{KabThm1}
For a complete bipartite graph with $k < a \leq b$, 
\begin{equation}\nonumber
     \bondnum{K_{a,b}}{k} \leq ka.
\end{equation}  
\end{thm}
\begin{proof}
    To generalize the strategy used in the proof of Theorem \ref{KabThm0}, one can remove $a-1$ edges from each $b_j$ for $j \in \{1,2, \ldots, k\}$ so that each $b_j$ is only adjacent to a unique $a_i$, for $i \in \{1,2, \ldots, k\}$.
    Also, remove the set of edges $\{a_ib_{k+1}\}$ for the same values of $i$.
\end{proof}

An example for this method is shown in Figure \ref{fig:Kab2} for $a=3$ and $b=6$.

\begin{figure}[h]
    \centering
   \begin{tikzpicture}
            \graph[nodes={draw, circle,minimum size = 1em},
           empty nodes,radius = 1.1cm, branch down=.8 cm,
           grow right sep=4cm] {subgraph I_nm [V={a, b, c}, W={1,...,6}];
  a -- { 1,4,5,6};
  b -- { 2,4,5,6 };
  c -- { 3,4,5,6 };
  
};
\node[tdvert, minimum size = 1em] at (a) {};
\node[tdvert, minimum size = 1em] at (b) {};
\node[tdvert, minimum size = 1em] at (c) {};
\node[tdvert, minimum size = 1em] at (4) {};
        \end{tikzpicture}
    \caption{$\bondnum{K_{3,6}}{2} = 6$}
    \label{fig:Kab2}
\end{figure}
We have found cases that show Theorem \ref{KabThm1} is not a strict bound, so we now give another method that in those cases improves on the above strategy. 

\begin{thm}\label{KabThm2}
For a complete bipartite graph with $ \lceil \frac{k}{2}\rceil +1 \leq a \leq b $,
\begin{eqnarray}\nonumber
     \bondnum{K_{a,b}}{k} \leq \left\lceil \frac{k}{2} \right\rceil\left(a+b-\left\lceil\frac{k}{2}\right\rceil-1\right).
\end{eqnarray}  
 
\end{thm}
\begin{proof}
    Consider removing the set of edges $\{a_1b_j \mid 2 \leq j \leq b \} \cup \{a_ib_1 \mid 2 \leq i \leq a \}$.
    This disconnects the vertices $a_1$ and $b_1$ from the rest of the graph, and will therefore increase the domination number by two.
    This strategy can be generalized by disconnecting $\left\lceil \frac{k}{2} \right\rceil$ pairs of vertices $\{a_i, b_i\}$ to increase the domination number by at least $k$, which requires removing $a+b-2i$ edges to isolate the $i\textsuperscript{{\rm th}}$ copy of $P_2$ from the graph. 
\end{proof}
\begin{figure}[h]
    \centering
    \begin{tikzpicture}
            \graph[nodes={draw, circle,minimum size = 1em},
           empty nodes,radius = 1.1cm, branch down=.8 cm,
           grow right sep=4cm] {subgraph I_nm [V={a, b, c,d}, W={1,...,5}];
  a -- { 1};
  b -- { 2,3,4,5 };
  c -- { 2,3,4,5 };
  d -- {2,3,4,5};
};
\node[tdvert, minimum size = 1em] at (a) {};
\node[tdvert, minimum size = 1em] at (b) {};
\node[tdvert, minimum size = 1em] at (1) {};
\node[tdvert, minimum size = 1em] at (2) {};
        \end{tikzpicture}
    \caption{$\bondnum{K_{4,5}}{2} = 7$}
    \label{fig:Kab3}
\end{figure}

Figure \ref{fig:Kab3} demonstrates the method from Theorem $\ref{KabThm2}$ for $a=4$ and $b=5$.
Theorems \ref{KabThm1} and \ref{KabThm2} yield different upper bounds on the total bondage number depending on the parity of $k$ and sizes of $a$ and $b$.
Theorem \ref{KabThm1} gives a stricter bound for even $k$ when $a  <  b - \frac{k}{2} - 1$.
For odd $k$ we have a similar result, that Theorem \ref{KabThm1} gives a stricter bound when $a  < b\left( \frac{k+1}{k-1}\right) - \frac{k-1}{2}$.
Likewise, when the sign is flipped Theorem \ref{KabThm2} gives a stricter bound.
In the case that the total bondage number is equal to either bound given by Theorem \ref{KabThm1} or \ref{KabThm2}, the bounds are equal.
Note that in some cases we would start with one strategy and then change as $k$ increases.

We now give one last strategy that only works for the case $\bondnum{K_{a,b}}{2}$. This strategy improves on Theorems \ref{KabThm1} and \ref{KabThm2} when $b<2a$. 

\begin{thm}\label{KabThm3}
    For a complete bipartite graph with $2 \leq a \leq b \leq 2a$, 
\begin{eqnarray}\nonumber
     \bondnum{K_{a,b}}{2} = b.
\end{eqnarray}  
\end{thm}
\begin{proof}

    Consider removing the edges $\{a_ib_i \mid 1 \leq i \leq a\} \cup \{a_ab_i \mid a+1 \leq i \leq b\}$.
    We must have two vertices from each part in the dominating set since each vertex is disconnected from at least one in the other part. 
    Therefore $\bondnum{K_{a,b}}{2} \leq b$.
    Suppose we remove only $b-1$ edges.
    Then there exists some $b_j$, say $b_1$, adjacent to all $a_i$. Note that each $a_i$ must be missing at least one edge, otherwise the total domination number is still two.  
    However, there must exist some $a_i$ that is missing only one edge since $b-1 < 2a$. Without loss of generality suppose $a_1$ is only missing the edge to $b_2$.
    Note that $b_2$ must be adjacent to some other vertex, say $a_2$, so $\{b_1,a_1,a_2\}$ dominates the graph. 
\end{proof}
\begin{figure}[h]
    \centering
   \begin{tikzpicture}
            \graph[nodes={draw, circle,minimum size = 1em},
           empty nodes,radius = 1.1cm, branch down=.8 cm,
           grow right sep=4cm] {subgraph I_nm [V={a, b, c}, W={1,...,5}];
  a -- { 2,3,4,5};
  b -- { 1,3,4,5 };
  c -- { 1,2 };
  
};
\node[tdvert, minimum size = 1em] at (a) {};
\node[tdvert, minimum size = 1em] at (b) {};
\node[tdvert, minimum size = 1em] at (1) {};
\node[tdvert, minimum size = 1em] at (2) {};
        \end{tikzpicture}
    \caption{$\bondnum{K_{3,5}}{2} = 5$}
    \label{fig:Kab4}
\end{figure}
Figure \ref{fig:Kab4} shows the method from Theorem $\ref{KabThm3}$ for $a=3$ and $b=5$.
\section{Graph Construction for $k$-total Bondage}
Kulli and Patwari showed in \cite{1991Bondage} that for any $b$ it is possible to construct a graph with $\bondnum{G}{1}=b$. In this section we show similar results for $\bondnum{G}{k}$.
\begin{thm}
\label{thm:btk_for_any_k_m}
    For any positive integers $k$ and $b \geq \frac{k}{2}$, there exists a graph $G^k_b$ where $\bondnum{G^k_b}{k} = b$.
\end{thm}
\begin{proof}
By Theorem \ref{thm:max_edge_removal_bondage_diff}, we can increase $\gamma_t(G)$ by at most two for each edge we remove, so $b \geq \frac{k}{2}$. Given two graphs $G$ and $H$ with $g \in V(G)$ and $h \in V(H)$ we will denote $G_g \oplus H_h$ as the merge of $G$ and $H$ at $g$ and $h$. 

Let $T = (P_5)_{v_3} \oplus (P_2)_{u_1}$ where $P_i = a_1a_2\ldots a_i$. Let $m = \lceil \frac{k-2}{2}\rceil$ and let $T^m$ denote the merge of $m$ copies of $T$ each at $v_1$ and let $x$ denote the merged vertices $v_1$. Let $T_{n}^m = (T^m)_x \oplus (K_{n}\setminus \{e\})_{a'}$, where $e = aa'$. Note that $\tdnum{T^m_{n}} = 2m+2$ and $\bondnum{T^m_{n}}{2i} = i$ for $1 \leq i \leq m$. All of the cases below start by removing $m$ edges (one from each copy of $T$), so let $r=b-m$ be the number of edges left to remove after this first step.\\

\begin{case}[1]
    Assume $k$ is odd. 
    For $r = 1$, consider $G_b^k = (T^m_{5})_x  \oplus (P_3)_{v_1}$. Note that $\tdnum{G_b^k} = 2m+3$.  We can increase the total dominating number by $k$ by removing one edge from each of the $m=\frac{k-1}{2}$ copies of $T$ as well as one edge from the $P_3$. This results in a graph $\tilde{G}^k_b$ which is the disjoint union of $m+1$ copies of $P_2$ with the merge of $m$ copies of $P_4$ and $K_{5} \setminus \{e\}$. The total dominating set must contain all the vertices in the $P_2$s as well as two vertices from each $P_4$ and two vertices from the $K_{5} \setminus \{e\}$. Therefore the new total dominating number is $\tdnum{\tilde{G}_b^k } = (2m+2) + 2m + 2 = 4m+4 = 2m+3 + k$. See Figure \ref{fig:kOddrOne} for details.

    For $r = 2$ consider $G_b^k = (T^m_{5})_x \oplus (K_4 \setminus \{e\})_{a'}$ where $e=aa'$. Note that $\tdnum{G_b^k} = 2m+3$.  We can increase the total domination number by $k-1$ by removing one edge from each of the $m=\frac{k-1}{2}$ copies of $T$. We can increase the total domination number one more by removing two edges from the $K_4 \setminus \{e\}$ component. See Figure \ref{fig:kOddrTwo} for details.

    For $r \geq 3$ consider $G_b^k = T^m_{r+1}$. Note that $\tdnum{G_b^k} = 2m+2$.  We can increase the total domination number by $k-1$ by removing one edge from each of the $m=\frac{k-1}{2}$ copies of $T$. We can increase the total domination number one more by removing $r-1$ edges to disconnect $x$ from $K_{r+1} \setminus \{e\}$, and one more edge to disconnect a $P_2$ from one of the copies of $T$. This is the fewest number of edges that can be removed to disconnect part of the complete graph, and without doing so, it is always possible to find a smaller total dominating set. See Figure \ref{fig:kOddrGEQThree} for details.

\end{case}

\begin{case}[2] Assume $k$ is even.

    For $r=1$ then $b = \frac{k}{2}$ and $\bondnum{T^{b}_{5}}{k} = b$.

    For $r = 2$ consider $G^b_k = (T^m_{5})_x \oplus (P_5)_{v_3}$.  

    For $r=3$ consider $G^b_k = (T^m_{5})_x  \oplus (P_3)_{v_1} \oplus (K_4 \setminus \{e\})_v$ where $e$ is an edge incident to $v$. 

    For $r \geq 4$ consider $G^b_k = (T^m_{r})_x \oplus (P_3)_{v_1}$.
\end{case}
\end{proof}
\begin{figure}[h]
    \begin{center}
        \begin{subfigure}[b]{0.45\textwidth}
            \centering
            \begin{tikzpicture}[scale=.8, transform shape]
                \node[tdvert, draw, circle, minimum size=0.3cm] (x) at (3,3) {$x$};
                \node[draw, circle, minimum size=0.3cm] (p_1) at (3,4.5) {};
                \node[tdvert, draw, circle, minimum size=0.3cm] (p_2) at (3,3.75) {};
                \node[tdvert, draw, circle, minimum size=0.3cm] (x) at (3,3) {$x$};
                \node[draw, circle, minimum size=0.3cm] (v_1) at (2,3.75) {};
                \node[tdvert, circle, minimum size=0.3cm] (v_2) at (1,3.5) {};
                \node[fill = white, draw, circle, minimum size=0.3cm] (v_3) at (2,2.25) {$a$};
                \node[draw, circle, minimum size=0.3cm] (v_4) at (1,2.5) {};
                \node[draw, circle, minimum size=0.3cm] (p1_1) at (3.75,3) {};
                \node[tdvert, draw, circle, minimum size=0.3cm] (p1_2) at (4.5,3) {};
                \node[draw, circle, minimum size=0.3cm] (p1_d) at (4.875,2.5) {};
                \node[tdvert, draw, circle, minimum size=0.3cm] (p1_3) at (5.25,3) {};
                \node[draw, circle, minimum size=0.3cm] (p1_4) at (6,3) {};
                \node[draw, circle, minimum size=0.3cm] (p2_1) at (3.75,4) {};
                \node[tdvert, draw, circle, minimum size=0.3cm] (p2_2) at (4.5,4) {};
                \node[draw, circle, minimum size=0.3cm] (p2_d) at (4.875,3.5) {};
                \node[tdvert, draw, circle, minimum size=0.3cm] (p2_3) at (5.25,4) {};
                \node[draw, circle, minimum size=0.3cm] (p2_4) at (6,4) {};
                \node[draw, circle, minimum size=0.3cm] (p3_1) at (3.75,2) {};
                \node[tdvert, draw, circle, minimum size=0.3cm] (p3_2) at (4.5,2) {};
                \node[draw, circle, minimum size=0.3cm] (p3_d) at (4.875,1.5) {};
                \node[tdvert, draw, circle, minimum size=0.3cm] (p3_3) at (5.25,2) {};
                \node[draw, circle, minimum size=0.3cm] (p3_4) at (6,2) {};
                    \graph{
                        (x) -- {(p1_1), (p2_1), (p3_1)};
                        (p1_1) -- (p1_2); (p1_3) -- (p1_4); (p1_2) -- (p1_d); (p1_2) -- (p1_3);
                        (p2_1) -- (p2_2); (p2_3) -- (p2_4); (p2_2) -- (p2_d); (p2_2) -- (p2_3);
                        (p3_1) -- (p3_2); (p3_3) -- (p3_4); (p3_2) -- (p3_d); (p3_2) -- (p3_3);
                        (p_1) -- (p_2); (p_2) -- (x);
                        (x)  -- {(v_1), (v_2), (v_4)};
                        (v_1) -- {(v_2), (v_3), (v_4)};
                        (v_2) -- {(v_3), (v_4)};
                        (v_3) -- {(v_4)};
                        };
            \end{tikzpicture}
            \caption{$G_4^7 = (T^3_{5})_x  \oplus (P_3)_{v_1}$}
        \end{subfigure}
        \begin{subfigure}[b]{0.45\textwidth}
            \centering
            \begin{tikzpicture}[scale=.8, transform shape]
                \node[tdvert, draw, circle, minimum size=0.3cm] (x) at (3,3) {$x$};
                \node[tdvert, draw, circle, minimum size=0.3cm] (p_1) at (3,4.5) {};
                \node[tdvert, draw, circle, minimum size=0.3cm] (p_2) at (3,3.75) {};
                \node[tdvert, draw, circle, minimum size=0.3cm] (x) at (3,3) {$x$};
                \node[draw, circle, minimum size=0.3cm] (v_1) at (2,3.75) {};
                \node[tdvert, circle, minimum size=0.3cm] (v_2) at (1,3.5) {};
                \node[fill = white, draw, circle, minimum size=0.3cm] (v_3) at (2,2.25) {$a$};
                \node[draw, circle, minimum size=0.3cm] (v_4) at (1,2.5) {};
                \node[tdvert, draw, circle, minimum size=0.3cm] (p1_1) at (3.75,3) {};
                \node[tdvert, draw, circle, minimum size=0.3cm] (p1_2) at (4.5,3) {};
                \node[draw, circle, minimum size=0.3cm] (p1_d) at (4.875,2.5) {};
                \node[tdvert, draw, circle, minimum size=0.3cm] (p1_3) at (5.25,3) {};
                \node[tdvert, draw, circle, minimum size=0.3cm] (p1_4) at (6,3) {};
                \node[tdvert, draw, circle, minimum size=0.3cm] (p2_1) at (3.75,4) {};
                \node[tdvert, draw, circle, minimum size=0.3cm] (p2_2) at (4.5,4) {};
                \node[draw, circle, minimum size=0.3cm] (p2_d) at (4.875,3.5) {};
                \node[tdvert, draw, circle, minimum size=0.3cm] (p2_3) at (5.25,4) {};
                \node[tdvert, draw, circle, minimum size=0.3cm] (p2_4) at (6,4) {};
                \node[tdvert, draw, circle, minimum size=0.3cm] (p3_1) at (3.75,2) {};
                \node[tdvert, draw, circle, minimum size=0.3cm] (p3_2) at (4.5,2) {};
                \node[draw, circle, minimum size=0.3cm] (p3_d) at (4.875,1.5) {};
                \node[tdvert, draw, circle, minimum size=0.3cm] (p3_3) at (5.25,2) {};
                \node[tdvert, draw, circle, minimum size=0.3cm] (p3_4) at (6,2) {};
                    \graph{
                        (x) -- {(p1_1), (p2_1), (p3_1)};
                        (p1_1) -- (p1_2); (p1_3) -- (p1_4); (p1_2) -- (p1_d); 
                        (p2_1) -- (p2_2); (p2_3) -- (p2_4); (p2_2) -- (p2_d); 
                        (p3_1) -- (p3_2); (p3_3) -- (p3_4); (p3_2) -- (p3_d); 
                        (p_1) -- (p_2); 
                        (x)  -- {(v_1), (v_2), (v_4)};
                        (v_1) -- {(v_2), (v_3), (v_4)};
                        (v_2) -- {(v_3), (v_4)};
                        (v_3) -- {(v_4)};
                    };
            \end{tikzpicture}
                \caption{$b_t^7(G_4^7) =4$}
                \end{subfigure}
            \caption{Case 1 when $k$ odd and $r =1$. For example $k = 7$ and $b = 4$} 
            \label{fig:kOddrOne}
    \end{center}
\end{figure}
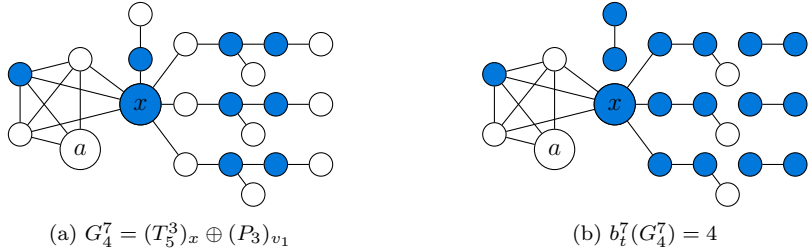

\begin{figure}[h]
    \begin{center}
        \begin{subfigure}[b]{0.45\textwidth}
            \centering
            \begin{tikzpicture}[scale=.8, transform shape]
                \node[tdvert, draw, circle, minimum size=0.3cm] (x) at (3,3) {$x$};
                \node[draw, circle, minimum size=0.3cm] (p_1) at (3,4.5) {};
                \node[tdvert, draw, circle, minimum size=0.3cm] (p_2) at (2.75,4) {};
                \node[draw, circle, minimum size=0.3cm] (p_3) at (3.25,4) {};
                \node[tdvert, draw, circle, minimum size=0.3cm] (x) at (3,3) {$x$};
                \node[draw, circle, minimum size=0.3cm] (v_1) at (2,3.75) {};
                \node[tdvert, circle, minimum size=0.3cm] (v_2) at (1,3.5) {};
                \node[fill = white, draw, circle, minimum size=0.3cm] (v_3) at (2,2.25) {$a$};
                \node[draw, circle, minimum size=0.3cm] (v_4) at (1,2.5) {}; 
                \node[draw, circle, minimum size=0.3cm] (p1_1) at (3.75,3) {};
                \node[tdvert, draw, circle, minimum size=0.3cm] (p1_2) at (4.5,3) {};
                \node[draw, circle, minimum size=0.3cm] (p1_d) at (4.875,2.5) {};
                \node[tdvert, draw, circle, minimum size=0.3cm] (p1_3) at (5.25,3) {};
                \node[draw, circle, minimum size=0.3cm] (p1_4) at (6,3) {};
                \node[draw, circle, minimum size=0.3cm] (p2_1) at (3.75,4) {};
                \node[tdvert, draw, circle, minimum size=0.3cm] (p2_2) at (4.5,4) {};
                \node[draw, circle, minimum size=0.3cm] (p2_d) at (4.875,3.5) {};
                \node[tdvert, draw, circle, minimum size=0.3cm] (p2_3) at (5.25,4) {};
                \node[draw, circle, minimum size=0.3cm] (p2_4) at (6,4) {};
                \node[draw, circle, minimum size=0.3cm] (p3_1) at (3.75,2) {};
                \node[tdvert, draw, circle, minimum size=0.3cm] (p3_2) at (4.5,2) {};
                \node[draw, circle, minimum size=0.3cm] (p3_d) at (4.875,1.5) {};
                \node[tdvert, draw, circle, minimum size=0.3cm] (p3_3) at (5.25,2) {};
                \node[draw, circle, minimum size=0.3cm] (p3_4) at (6,2) {};
                    \graph{
                        (x) -- {(p1_1), (p2_1), (p3_1)};
                        (p1_1) -- (p1_2); (p1_3) -- (p1_4); (p1_2) -- (p1_d); (p1_2) -- (p1_3);
                        (p2_1) -- (p2_2); (p2_3) -- (p2_4); (p2_2) -- (p2_d); (p2_2) -- (p2_3);
                        (p3_1) -- (p3_2); (p3_3) -- (p3_4); (p3_2) -- (p3_d); (p3_2) -- (p3_3);
                        (p_1) -- (p_2); (p_2) -- (x);
                        (p_3) -- (x); (p_3) -- (p_1);
                        (x)  -- {(v_1), (v_2), (v_4)};
                        (v_1) -- {(v_2), (v_3), (v_4)};
                        (v_2) -- {(v_3), (v_4)};
                        (v_3) -- {(v_4)};
                        
                    };
            \end{tikzpicture}
            \caption{$G_5^7 = (T^3_{5})_x \oplus (K_4 \setminus \{e\})_{a'}$}
        \end{subfigure}
        \begin{subfigure}[b]{0.45\textwidth}
            \centering
            \begin{tikzpicture}[scale=.8, transform shape]
                \node[tdvert, draw, circle, minimum size=0.3cm] (x) at (3,3) {$x$};
                \node[tdvert, draw, circle, minimum size=0.3cm] (p_1) at (3,4.5) {};
                \node[tdvert, draw, circle, minimum size=0.3cm] (p_2) at (2.75,4) {};
                \node[draw, circle, minimum size=0.3cm] (p_3) at (3.25,4) {};

                \node[tdvert, draw, circle, minimum size=0.3cm] (x) at (3,3) {$x$};
                \node[draw, circle, minimum size=0.3cm] (v_1) at (2,3.75) {};
                \node[tdvert, circle, minimum size=0.3cm] (v_2) at (1,3.5) {};
                \node[fill = white, draw, circle, minimum size=0.3cm] (v_3) at (2,2.25) {$a$};
                \node[draw, circle, minimum size=0.3cm] (v_4) at (1,2.5) {};
                \node[tdvert, draw, circle, minimum size=0.3cm] (p1_1) at (3.75,3) {};
                \node[tdvert, draw, circle, minimum size=0.3cm] (p1_2) at (4.5,3) {};
                \node[draw, circle, minimum size=0.3cm] (p1_d) at (4.875,2.5) {};
                \node[tdvert, draw, circle, minimum size=0.3cm] (p1_3) at (5.25,3) {};
                \node[tdvert, draw, circle, minimum size=0.3cm] (p1_4) at (6,3) {};
                \node[tdvert, draw, circle, minimum size=0.3cm] (p2_1) at (3.75,4) {};
                \node[tdvert, draw, circle, minimum size=0.3cm] (p2_2) at (4.5,4) {};
                \node[draw, circle, minimum size=0.3cm] (p2_d) at (4.875,3.5) {};
                \node[tdvert, draw, circle, minimum size=0.3cm] (p2_3) at (5.25,4) {};
                \node[tdvert, draw, circle, minimum size=0.3cm] (p2_4) at (6,4) {};
                \node[tdvert, draw, circle, minimum size=0.3cm] (p3_1) at (3.75,2) {};
                \node[tdvert, draw, circle, minimum size=0.3cm] (p3_2) at (4.5,2) {};
                \node[draw, circle, minimum size=0.3cm] (p3_d) at (4.875,1.5) {};
                \node[tdvert, draw, circle, minimum size=0.3cm] (p3_3) at (5.25,2) {};
                \node[tdvert, draw, circle, minimum size=0.3cm] (p3_4) at (6,2) {};
                    \graph{
                        (x) -- {(p1_1), (p2_1), (p3_1)};
                        (p1_1) -- (p1_2); (p1_3) -- (p1_4); (p1_2) -- (p1_d); 
                        (p2_1) -- (p2_2); (p2_3) -- (p2_4); (p2_2) -- (p2_d); 
                        (p3_1) -- (p3_2); (p3_3) -- (p3_4); (p3_2) -- (p3_d); 
                        (p_1) -- (p_2); (p_3) -- (p_1); 
                        (x)  -- {(v_1), (v_2), (v_4)};
                        (v_1) -- {(v_2), (v_3), (v_4)};
                        (v_2) -- {(v_3), (v_4)};
                        (v_3) -- {(v_4)};
                    };
            \end{tikzpicture}
                \caption{$b_t^7(G_5^7) =5$}
                \end{subfigure}
            \caption{Case 1 when $k$ odd and $r =2$. For example $k = 7$ and $b = 5$} 
            \label{fig:kOddrTwo}
    \end{center}
\end{figure}
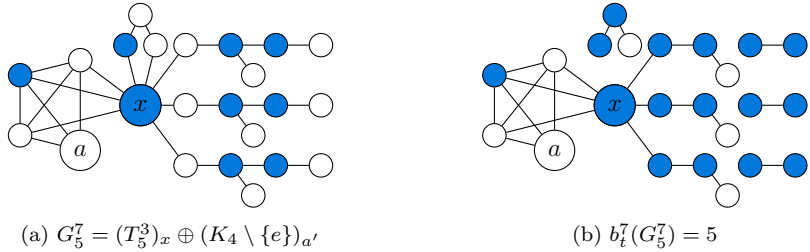

\begin{figure}[h]
    \begin{center}
        \begin{subfigure}[b]{0.45\textwidth}
            \centering
            \begin{tikzpicture}[scale=.8, transform shape]
                \node[tdvert, draw, circle, minimum size=0.3cm] (x) at (3,3) {$x$};
                \node[draw, circle, minimum size=0.3cm] (v_1) at (2.25,3.75) {};
                \node[draw, circle, minimum size=0.3cm] (v_2) at (1.5,3.75) {};
                \node[tdvert, draw, circle, minimum size=0.3cm] (v_3) at (0.75,3) {};
                \node[draw, circle, minimum size=0.3cm] (v_4) at (1.5,2.25) {};
                \node[fill = white, draw, circle, minimum size=0.3cm] (v_5) at (2.25,2.25) {$a$};
                \node[draw, circle, minimum size=0.3cm] (p1_1) at (3.75,3) {};
                \node[tdvert, draw, circle, minimum size=0.3cm] (p1_2) at (4.5,3) {};
                \node[draw, circle, minimum size=0.3cm] (p1_d) at (4.875,2.5) {};
                \node[tdvert, draw, circle, minimum size=0.3cm] (p1_3) at (5.25,3) {};
                \node[draw, circle, minimum size=0.3cm] (p1_4) at (6,3) {};
                \node[draw, circle, minimum size=0.3cm] (p2_1) at (3.75,4) {};
                \node[tdvert, draw, circle, minimum size=0.3cm] (p2_2) at (4.5,4) {};
                \node[draw, circle, minimum size=0.3cm] (p2_d) at (4.875,3.5) {};
                \node[tdvert, draw, circle, minimum size=0.3cm] (p2_3) at (5.25,4) {};
                \node[draw, circle, minimum size=0.3cm] (p2_4) at (6,4) {};
                \node[draw, circle, minimum size=0.3cm] (p3_1) at (3.75,2) {};
                \node[tdvert, draw, circle, minimum size=0.3cm] (p3_2) at (4.5,2) {};
                \node[draw, circle, minimum size=0.3cm] (p3_d) at (4.875,1.5) {};
                \node[tdvert, draw, circle, minimum size=0.3cm] (p3_3) at (5.25,2) {};
                \node[draw, circle, minimum size=0.3cm] (p3_4) at (6,2) {};
                    \graph{
                        (x) -- {(p1_1), (p2_1), (p3_1)};
                        (p1_1) -- (p1_2); (p1_3) -- (p1_4); (p1_2) -- (p1_d); (p1_2) -- (p1_3);
                        (p2_1) -- (p2_2); (p2_3) -- (p2_4); (p2_2) -- (p2_d); (p2_2) -- (p2_3);
                        (p3_1) -- (p3_2); (p3_3) -- (p3_4); (p3_2) -- (p3_d); (p3_2) -- (p3_3);
                        (x) -- {(v_1), (v_2), (v_3), (v_4)};
                        (v_1) -- {(v_2), (v_3), (v_4), (v_5)};
                        (v_2) -- {(v_3), (v_4), (v_5)};
                        (v_3) -- {(v_4), (v_5)};
                        (v_4) -- (v_5);
                    };
            \end{tikzpicture}
            \caption{$G_8^7 = T^3_{6}$}
        \end{subfigure}
        \begin{subfigure}[b]{0.45\textwidth}
            \centering
            \begin{tikzpicture}[scale=.8, transform shape]
                \node[draw, circle, minimum size=0.3cm] (v_1) at (2.25,3.75) {};
                \node[draw, circle, minimum size=0.3cm] (v_2) at (1.5,3.75) {};
                \node[tdvert, draw, circle, minimum size=0.3cm] (v_3) at (0.75,3) {};
                \node[draw, circle, minimum size=0.3cm] (v_4) at (1.5,2.25) {};
                \node[tdvert, draw, circle, minimum size=0.3cm] (v_5) at (2.25,2.25) {$a$};
                 \node[tdvert, draw, circle, minimum size=0.3cm] (x) at (3,3) {$x$};
                \node[tdvert, draw, circle, minimum size=0.3cm] (p1_1) at (3.75,3) {};
                \node[tdvert, draw, circle, minimum size=0.3cm] (p1_2) at (4.5,3) {};
                \node[draw, circle, minimum size=0.3cm] (p1_d) at (4.875,2.5) {};
                \node[tdvert, draw, circle, minimum size=0.3cm] (p1_3) at (5.25,3) {};
                \node[tdvert, draw, circle, minimum size=0.3cm] (p1_4) at (6,3) {};
                \node[tdvert, draw, circle, minimum size=0.3cm] (p2_1) at (3.75,4) {};
                \node[tdvert, draw, circle, minimum size=0.3cm] (p2_2) at (4.5,4) {};
                \node[draw, circle, minimum size=0.3cm] (p2_d) at (4.875,3.5) {};
                \node[tdvert, draw, circle, minimum size=0.3cm] (p2_3) at (5.25,4) {};
                \node[tdvert, draw, circle, minimum size=0.3cm] (p2_4) at (6,4) {};
                \node[draw, circle, minimum size=0.3cm] (p3_1) at (3.75,2) {};
                \node[tdvert, draw, circle, minimum size=0.3cm] (p3_2) at (4.5,2) {};
                \node[tdvert, draw, circle, minimum size=0.3cm] (p3_d) at (4.875,1.5) {};
                \node[tdvert, draw, circle, minimum size=0.3cm] (p3_3) at (5.25,2) {};
                \node[tdvert, draw, circle, minimum size=0.3cm] (p3_4) at (6,2) {};
                    \graph{
                        (x) -- {(p1_1), (p2_1), (p3_1)};
                        (p1_1) -- (p1_2); (p1_3) -- (p1_4); (p1_2) -- (p1_d); 
                        (p2_1) -- (p2_2); (p2_3) -- (p2_4); (p2_2) -- (p2_d); 
                        (p3_3) -- (p3_4); (p3_2) -- (p3_d); 
                        (v_1) -- {(v_2), (v_3), (v_4), (v_5)};
                        (v_2) -- {(v_3), (v_4), (v_5)};
                        (v_3) -- {(v_4), (v_5)};
                        (v_4) -- (v_5);
                        };
            \end{tikzpicture}
                \caption{$b_t^7(G_8^7) =8$}
                \end{subfigure}
            \caption{Case 1 when $k$ odd and $r =5$. For example $k = 7$ and $b = 8$} 
            \label{fig:kOddrGEQThree}
    \end{center}
\end{figure}
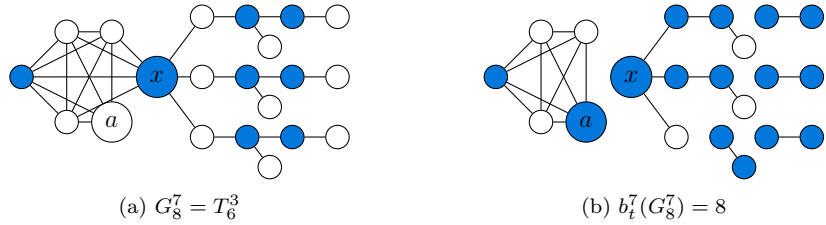

\begin{cor}
    For any $k$ and $b$, there exists a graph $G$ where $\bondnum{G}{k+1} =  \bondnum{G}{k} + b$. 
\end{cor}

\begin{proof}
    The result follows from Theorem \ref{thm:btk_for_any_k_m}. Let $m = \lfloor \frac{k}{2}\rfloor.$
    When $k$ is even, consider $G = T^m_{b+1}$ and when $k$ is odd, consider $G = (T^m_{b+1})_x \oplus (P_3)_{v_1}$.
\end{proof}


\begin{thebibliography}{10}
\expandafter\ifx\csname url\endcsname\relax
  \def\url#1{\texttt{#1}}\fi
\expandafter\ifx\csname urlprefix\endcsname\relax\def\urlprefix{URL }\fi
\expandafter\ifx\csname href\endcsname\relax
  \def\href#1#2{#2} \def\path#1{#1}\fi

\bibitem{berge1958}
C.~Berge, Th\'eorie des graphes et ses applications, Vol.~II of Collection
  Universitaire de Math\'ematiques, Dunod, Paris, 1958.

\bibitem{ore1962}
O.~Ore, Theory of graphs, Vol. Vol. XXXVIII of American Mathematical Society
  Colloquium Publications, American Mathematical Society, Providence, RI, 1962.

\bibitem{Cockayne1975}
E.~Cockayne, S.~Hedetniemi,
  \href{https://doi.org/10.1109/TCS.1975.1083994}{Optimal domination in
  graphs}, IEEE Transactions Circuits and Systems CAS-2~(11) (1975) 855--857.
\newblock \href {https://doi.org/10.1109/TCS.1975.1083994}
  {\path{doi:10.1109/TCS.1975.1083994}}.
\newline\urlprefix\url{https://doi.org/10.1109/TCS.1975.1083994}

\bibitem{Cockayne1976}
E.~Cockayne, S.~Hedetniemi,
  \href{https://doi.org/10.1016/0012-365X(76)90026-1}{Disjoint independent
  dominating sets in graphs}, Discrete Mathematics 15~(3) (1976) 213--222.
\newblock \href {https://doi.org/10.1016/0012-365X(76)90026-1}
  {\path{doi:10.1016/0012-365X(76)90026-1}}.
\newline\urlprefix\url{https://doi.org/10.1016/0012-365X(76)90026-1}

\bibitem{Cockayne1977}
E.~Cockayne, S.~Hedetniemi,
  \href{https://doi.org/10.1002/net.3230070305}{Towards a theory of domination
  in graphs}, Networks 7~(3) (1977) 247--261.
\newblock \href {https://doi.org/10.1002/net.3230070305}
  {\path{doi:10.1002/net.3230070305}}.
\newline\urlprefix\url{https://doi.org/10.1002/net.3230070305}

\bibitem{Cockayne1978}
E.~Cockayne, Domination of undirected graphs---a survey, Theory and
  applications of graphs ({P}roc. {I}nternat. {C}onf., {W}estern {M}ich.
  {U}niv., {K}alamazoo, {M}ich., 1976) Vol. 642 (1978) 141--147.

\bibitem{Du2013}
D.~Du, P.~Wan, \href{https://doi.org/10.1007/978-1-4614-5242-3}{Connected
  dominating set: theory and applications}, Vol.~77 of Springer Optimization
  and Its Applications, Springer, New York, 2013.
\newblock \href {https://doi.org/10.1007/978-1-4614-5242-3}
  {\path{doi:10.1007/978-1-4614-5242-3}}.
\newline\urlprefix\url{https://doi.org/10.1007/978-1-4614-5242-3}

\bibitem{haynes1998}
T.~Haynes, S.~Hedetniemi, P.~J. Slater, Fundamentals of Domination in Graphs,
  1st Edition, CRC Press, 1998.
\newblock \href {https://doi.org/10.1201/9781482246582}
  {\path{doi:10.1201/9781482246582}}.

\bibitem{Kelleher1985}
L.~L. Kelleher,
  \href{http://gateway.proquest.com/openurl?url_ver=Z39.88-2004&rft_val_fmt=info:ofi/fmt:kev:mtx:dissertation&res_dat=xri:pqdiss&rft_dat=xri:pqdiss:8528145}{Domination
  in graphs and its application to social network theory}, ProQuest LLC, Ann
  Arbor, MI, 1985, thesis (Ph.D.)--Northeastern University.
\newline\urlprefix\url{http://gateway.proquest.com/openurl?url_ver=Z39.88-2004&rft_val_fmt=info:ofi/fmt:kev:mtx:dissertation&res_dat=xri:pqdiss&rft_dat=xri:pqdiss:8528145}

\bibitem{Totaldom1980}
E.~J. Cockayne, R.~M. Dawes, S.~T. Hedetniemi,
  \href{https://doi.org/10.1002/net.3230100304}{Total domination in graphs},
  Networks 10~(3) (1980) 211--219.
\newblock \href {https://doi.org/10.1002/net.3230100304}
  {\path{doi:10.1002/net.3230100304}}.
\newline\urlprefix\url{https://doi.org/10.1002/net.3230100304}

\bibitem{TotalDomTextbook}
M.~A. Henning, A.~Yeo, \href{https://doi.org/10.1007/978-1-4614-6525-6}{Total
  domination in graphs}, Springer Monographs in Mathematics, Springer, New
  York, 2013.
\newblock \href {https://doi.org/10.1007/978-1-4614-6525-6}
  {\path{doi:10.1007/978-1-4614-6525-6}}.
\newline\urlprefix\url{https://doi.org/10.1007/978-1-4614-6525-6}

\bibitem{bauer1983domination}
D.~Bauer, F.~Harary, J.~Nieminen, C.~Suffel, Domination alteration sets in
  graphs, Discrete Mathematics 47 (1983) 153--161.
\newblock \href {https://doi.org/10.1016/0012-365X(83)90085-7}
  {\path{doi:10.1016/0012-365X(83)90085-7}}.

\bibitem{Bondage1990}
J.~F. Fink, M.~S. Jacobson, L.~F. Kinch, J.~Roberts,
  \href{https://doi.org/10.1016/0012-365X(90)90348-L}{The bondage number of a
  graph}, Discrete Mathematics 86~(1-3) (1990) 47--57.
\newblock \href {https://doi.org/10.1016/0012-365X(90)90348-L}
  {\path{doi:10.1016/0012-365X(90)90348-L}}.
\newline\urlprefix\url{https://doi.org/10.1016/0012-365X(90)90348-L}

\bibitem{bondage_survey}
J.~M. Xu, \href{https://onlinelibrary.wiley.com/doi/abs/10.1155/2013/595210}{On
  bondage numbers of graphs: A survey with some comments}, International
  Journal of Combinatorics 2013~(1) (2013) 595210.
\newblock \href {https://doi.org/https://doi.org/10.1155/2013/595210}
  {\path{doi:https://doi.org/10.1155/2013/595210}}.
\newline\urlprefix\url{https://onlinelibrary.wiley.com/doi/abs/10.1155/2013/595210}

\bibitem{anaya2025}
R.~Anaya, A.~Belmonte, N.~Shank, E.~Sinani, B.~Walker,
  \href{https://arxiv.org/abs/2208.07484}{Generalized bondage number: The
  $k$-synchronous bondage number of a graph}, Ball State Undergraduate
  Mathematics Exchange (to appear, 2025).
\newline\urlprefix\url{https://arxiv.org/abs/2208.07484}

\bibitem{1991Bondage}
V.~R. Kulli, D.~K. Patwari, The total bondage number of a graph, Advances in
  graph theory (1991) 227--235\href
  {https://doi.org/https://www.researchgate.net/profile/V-Kulli/publication/293504205_The_total_bondage_number_of_a_graph/links/56fe8b5a08ae650a64f72094/The-total-bondage-number-of-a-graph.pdf}
  {\path{doi:https://www.researchgate.net/profile/V-Kulli/publication/293504205_The_total_bondage_number_of_a_graph/links/56fe8b5a08ae650a64f72094/The-total-bondage-number-of-a-graph.pdf}}.

\bibitem{sanchis2004relating}
L.~A. Sanchis, Relating the size of a connected graph to its total and
  restricted domination numbers, Discrete Mathematics 283~(1-3) (2004)
  205--216.
\newblock \href {https://doi.org/https://doi.org/10.1016/j.disc.2003.11.011}
  {\path{doi:https://doi.org/10.1016/j.disc.2003.11.011}}.

\end{thebibliography}
\end{document}